\documentclass[a4paper, 10pt]{article}

\usepackage[T1]{fontenc}
\usepackage[latin1]{inputenc}
\usepackage{lmodern}  
\usepackage[english]{babel}
\usepackage{csquotes} 

\usepackage[scale=0.73, a4paper]{geometry}
\usepackage{textcomp} 
\usepackage{graphicx} 
\usepackage{amsmath} 
\usepackage{amsthm} 
\usepackage{amssymb} 
\usepackage{amsfonts}
\usepackage{mathtools}
\usepackage{thmtools} 
\usepackage{mathrsfs} 
\usepackage{mathdots} 
\usepackage{wasysym} 
\usepackage{extpfeil} 
\usepackage{verbatim} 
\usepackage{longtable} 
\usepackage{bigstrut} 
\usepackage[table, dvipsnames]{xcolor} 
\usepackage{floatpag} 
\usepackage{enumitem} 
\usepackage{xparse} 
\usepackage{mathscinet} 
\usepackage{adjustbox} 
\usepackage[obeyDraft]{todonotes} 

\usepackage[final,                  
            colorlinks,             
            plainpages = false,     
            linktoc=all,            
            linkcolor=black,        
            citecolor=black,         
            urlcolor=black,          
            filecolor=black         
            ]{hyperref}

\usepackage[citestyle=alphabetic,   
            bibstyle=alphabetic,    
            maxbibnames=50,         
            maxcitenames=3,         
            autocite=inline,        
            block=space,            
            hyperref=true,          
            backref=false,           
            backrefstyle=three,     
            date=short,             
            arxiv=pdf,              
            isbn=false,             
            url=false,              
            doi=false,               
            eprint=true,            
            giveninits=true,        
            sorting=nyt,            
            backend=bibtex8,
            ]{biblatex}

\bibliography{ind}

\hypersetup{pdftitle = {On the Induction of p-Cells},
            pdfauthor = {Lars Thorge Jensen, Leonardo Patimo},
            pdfkeywords = {p-canonical basis} {Hecke algebra} {p-Cells} {parity sheaves},
            pdfdisplaydoctitle      
            }
            
\usepackage[arrow, matrix, curve]{xy} 
\usepackage{tikz,tikz-cd} 
\usetikzlibrary{arrows, arrows.meta, bending, fit, intersections, calc, external, shapes.misc, shapes.geometric, decorations.markings, decorations.pathreplacing}
\usepackage{pgfplots}
\pgfplotsset{compat=1.14}

\allowdisplaybreaks 

\setkeys{Gin}{draft=false}
            
\declaretheoremstyle[
spaceabove=\topsep, spacebelow=\topsep,
headfont=\normalfont\bfseries,
notefont=\mdseries, notebraces={(}{)},
bodyfont=\itshape,
postheadspace=\newline
]{break}

\declaretheoremstyle[
spaceabove=\topsep, spacebelow=\topsep,
headfont=\normalfont\bfseries,
notefont=\mdseries, notebraces={}{},
bodyfont=\itshape,
postheadspace=\newline
]{refbreak}

\declaretheorem[title=Theorem, style=plain, numberwithin=section]{thm}
\declaretheorem[title=Proposition, style=plain, numberlike=thm]{prop}
\declaretheorem[title=Lemma, style=plain, numberlike=thm]{lem}
\declaretheorem[title=Corollary, style=plain, numberlike=thm]{cor}

\declaretheorem[title=Theorem, style=plain, numbered=no]{thm*}

\declaretheorem[title=Theorem, style=break, numberlike=thm]{thmlab}

\declaretheorem[title=Definition, style=definition, numberlike=thm]{defn}

\declaretheorem[title=Remark, style=remark, numberlike=thm]{remark}

\declaretheorem[title=Example, style=remark, numberlike=thm]{ex}

\usepackage[capitalize]{cleveref} 
                          
\crefname{thm}{Theorem}{Theorems}
\crefname{prop}{Proposition}{Propositions}
\crefname{lem}{Lemma}{Lemmata}
\crefname{cor}{Corollary}{Corollaries}
\crefname{rem}{Reminder}{Reminders}
\crefname{defn}{Definition}{Definitions}

\crefname{thmlab}{Theorem}{Theorems}
\crefname{proplab}{Proposition}{Propositions}
\crefname{lemlab}{Lemma}{Lemmata}
\crefname{corlab}{Corollary}{Corollaries}
\crefname{remlab}{Reminder}{Reminders}
\crefname{conj}{Conjecture}{Conjectures}

\crefname{thmreflab}{Theorem}{Theorems}
\crefname{propreflab}{Proposition}{Propositions}
\crefname{lemreflab}{Lemma}{Lemmata}
\crefname{correflab}{Corollary}{Corollaries}
\crefname{remreflab}{Reminder}{Reminders}
\crefname{conjref}{Conjecture}{Conjectures}

\crefname{remark}{Remark}{Remarks}
\crefname{claim}{Claim}{Claims}
\crefname{ex}{Example}{Examples}

\crefname{section}{Section}{Sections}
\crefname{figure}{Figure}{Figures}
\crefname{equation}{}{}
\crefname{ass}{Assumption}{Assumptions}

\CompileMatrices

\entrymodifiers={+!!<0pt,\fontdimen22\textfont2>} 

\def\clap#1{\hbox to 0pt{\hss#1\hss}}

\makeatletter
\def\underbracket{%
    \@ifnextchar[{\@underbracket}{\@underbracket [\@bracketheight]}%
}
\def\@underbracket[#1]{%
    \@ifnextchar[{\@under@bracket[#1]}{\@under@bracket[#1][0.4em]}%
}
\def\@under@bracket[#1][#2]#3{
    \mathop{\vtop{\m@th \ialign {##\crcr $\hfil \displaystyle {#3}\hfil $%
    \crcr \noalign {\kern 3\p@ \nointerlineskip }\upbracketfill {#1}{#2}
    \crcr \noalign {\kern 3\p@ }}}}\limits}

\def\upbracketfill#1#2{$\m@th \setbox \z@ \hbox {$\braceld$}
    \edef\@bracketheight{\the\ht\z@}\bracketend{#1}{#2}
    \leaders \vrule \@height #1 \@depth \z@ \hfill
    \leaders \vrule \@height #1 \@depth \z@ \hfill \bracketend{#1}{#2}$}

\def\bracketend#1#2{\vrule height #2 width #1\relax}
\makeatother

\makeatletter
\def\thmt@refnamewithcomma #1#2#3,#4,#5\@nil{%
  \@xa\def\csname\thmt@envname #1utorefname\endcsname{#3}%
  \ifcsname #2refname\endcsname
    \csname #2refname\expandafter\endcsname\expandafter{\thmt@envname}{#3}{#4}%
  \fi
}
\makeatother

\makeatletter
\newcommand*\rel@kern[1]{\kern#1\dimexpr\macc@kerna}
\newcommand*\widebar[1]{%
  \begingroup
  \def\mathaccent##1##2{%
    \rel@kern{0.8}%
    \overline{\rel@kern{-0.8}\macc@nucleus\rel@kern{0.2}}%
    \rel@kern{-0.2}%
  }%
  \macc@depth\@ne
  \let\math@bgroup\@empty \let\math@egroup\macc@set@skewchar
  \mathsurround\z@ \frozen@everymath{\mathgroup\macc@group\relax}%
  \macc@set@skewchar\relax
  \let\mathaccentV\macc@nested@a
  \macc@nested@a\relax111{#1}%
  \endgroup
}
\makeatother

\makeatletter
\newcommand{\subjclass}[2][1991]{%
  \let\@oldtitle\@title%
  \gdef\@title{\@oldtitle\footnotetext{#1 \emph{Mathematics subject classification.} #2}}%
}
\newcommand{\keywords}[1]{%
  \let\@@oldtitle\@title%
  \gdef\@title{\@@oldtitle\footnotetext{\emph{Key words and phrases.} #1.}}%
}
\makeatother

\makeatletter
\newcommand{\extp}{\@ifnextchar^\@extp{\@extp^{\,}}}
   \def\@extp^#1{\mathop{\bigwedge\nolimits^{\!#1}}}
\makeatother

\DeclareMathOperator{\Hom}{Hom}
\DeclareMathOperator{\End}{End}

\DeclareMathOperator{\id}{Id}

\DeclareMathOperator{\pr}{pr}

\DeclareMathOperator{\ch}{ch}

\DeclareMathOperator{\rk}{rk}

\DeclareMathOperator{\f}{f}

\newcommand{\defeq}{\ensuremath{\coloneqq}}

\newcommand{\cat}[1]{\ensuremath{\mathcal{#1}}}

\newcommand{\G}[1]{\ensuremath{[\cat{#1}]}}

\newcommand{\Z}{\ensuremath{\mathbb{Z}}}
\newcommand{\N}{\ensuremath{\mathbb{N}}}
\newcommand{\Q}{\ensuremath{\mathbb{Q}}}
\newcommand{\R}{\ensuremath{\mathbb{R}}}
\newcommand{\C}{\ensuremath{\mathbb{C}}}
\newcommand{\F}[1]{\ensuremath{\mathbb{F}_{#1}}}
\newcommand{\K}{\ensuremath{\mathbb{K}}}

\newcommand{\p}[1]{\ensuremath{{}^{p} #1}}
\newcommand{\pre}[2]{\ensuremath{{}^{#1} #2}}

\newcommand{\heck}[1][]{\ensuremath{\mathcal{H}_{#1}}}

\newcommand{\std}[2][H]{\ensuremath{#1_{#2}}}
\newcommand{\kl}[2][H]{\ensuremath{\underline{#1}_{#2}}}
\DeclareDocumentCommand{\pcan}{O{p} O{H} m}{\ensuremath{{}^{#1}\underline{#2}_{#3}}}

\DeclareDocumentCommand{\mixCan}{O{p} O{H} m m}{\ensuremath{{}^{#1}\underline{#2}_{#3}{#2}_{#4}}}

\newcommand{\cle}[2][p]{\ensuremath{\, \overset{#1}{\underset{#2}{\leqslant}} \,}}

\newcommand{\clt}[2][p]{\ensuremath{\, \overset{#1}{\underset{#2}{<}} \,}}

\newcommand{\ceq}[2][p]{\ensuremath{\, \overset{#1}{\underset{#2}{\sim}} \,}}
\newcommand{\mle}[2][p]{\ensuremath{\, \overset{#1}{\underset{#2}{\sqsubseteq}} \,}}
\newcommand{\mlt}[2][p]{\ensuremath{\, \overset{#1}{\underset{#2}{\sqsubset}} \,}}

\newcommand{\T}{\ensuremath{\cat{T}}}

\newcommand{\rt}[1][]{\ensuremath{\alpha_{#1}}}
\newcommand{\cort}[1][]{\ensuremath{\alpha_{#1}^{\vee}}}
\newcommand{\rts}{\ensuremath{\Phi}}

\newcommand{\charlat}{\ensuremath{X}}
\newcommand{\cocharlat}{\ensuremath{X^{\vee}}}

\newcommand{\Wid}{\ensuremath{\id}}

\newcommand{\Fl}[1][]{\ensuremath{\mathcal{F}l_{#1}}}

\newcommand{\parity}[1][]{\ensuremath{\mathbf{Parity}_{#1}}}

\newcommand{\Homnl}[2][]{
   \ifthenelse{ \equal{#1}{} } { \ensuremath{\Hom_{\nless #2}} } 
   { \ensuremath{\Hom_{\nless #2, #1}} } }
\newcommand{\Endnl}[2][]{
   \ifthenelse{ \equal{#1}{} } { \ensuremath{\End_{\nless #2}} } 
   { \ensuremath{\End_{\nless #2, #1}} } }
\newcommand{\E}[2][]{\ensuremath{\mathcal{E}^{#1}_{#2}}}
\newcommand{\D}{\ensuremath{\mathbb{D}}}

\renewcommand{\cal}[1]{\mathcal{#1}}

\newcommand{\llrrbracket}[1]{
  [\![#1]\!]}
  
\newcommand{\llrrparen}[1]{
  (\!(#1)\!)}

\newcommand{\Address}{
  \bigskip{\footnotesize
  \textsc{Universit\'{e} Clermont Auvergne, CNRS, LMBP, F-63178 Aubi\`ere, France}\par\nopagebreak
  \textit{E-mail address}, Lars~Thorge~Jensen: \texttt{lars\_thorge.jensen@uca.fr}
}

  \bigskip{\footnotesize
  \textsc{Albert-Ludwigs-Universit\"at Freiburg, D-79104 Freiburg im Breisgau, Germany}\par\nopagebreak
  \textit{E-mail address}, Leonardo~Patimo: \texttt{leonardo.patimo@math.uni-freiburg.de}
}}

\tikzset{%
  DynNode/.style={circle, inner sep=2pt, draw=black, fill=white},
  Greater/.style={pos=0.65, inner sep=0mm, outer sep=0mm},
  highlight/.style={rectangle,rounded corners,fill=red!15,draw=red,
    fill opacity=0.5,thick},
  root/.style={draw, color=black, thick, ->},
  plane/.style={draw, color=black, very thin},
  origin/.style={fill, color=black},
  sline/.style={color=Red, thick},
  tline/.style={color=NavyBlue, thick},
  uline/.style={color=Goldenrod, thick},
  bendBelow/.style={bend left=70, looseness=2},
  bendAbove/.style={bend right=70, looseness=2},
  object/.style={circle, fill, inner sep=1.5pt, outer sep=0mm},
  labelling/.style={outer sep=0mm, inner sep=0mm},
  1morph/.style={->, shorten >= 0.5pt, >=stealth'},
  2morph/.style={-implies,double,double equal sign distance,
                 shorten >=2pt, shorten <=3pt},
  spot/.style={color=black, thin, dashed},
  s/.style={color=Red},
  t/.style={color=NavyBlue},
  u/.style={color=Goldenrod},
  line/.style={draw, line width=2pt},
  dot/.style={fill, thin},
  sph/.style={fill, color=black!20, opacity=0.5},
  squig/.style={decoration=snake, decorate, ->},
  root/.style={draw, color=black, line width=2pt, ->},
  myptr/.style={decoration={markings,mark=at position 1 with %
    {\arrow[scale=3,>=stealth]{>}}},postaction={decorate}},
  on each segment/.style={
    decorate,
    decoration={
      show path construction,
      moveto code={},
      lineto code={
        \path [#1]
        (\tikzinputsegmentfirst) -- (\tikzinputsegmentlast);
      },
      curveto code={
        \path [#1] (\tikzinputsegmentfirst)
        .. controls
        (\tikzinputsegmentsupporta) and (\tikzinputsegmentsupportb)
        ..
        (\tikzinputsegmentlast);
      },
      closepath code={
        \path [#1]
        (\tikzinputsegmentfirst) -- (\tikzinputsegmentlast);
      },
    },
  },
  mid arrow/.style={postaction={decorate,decoration={
        markings,
        mark=at position .5 with {\arrow[#1]{stealth}}
      }}},
  sline/.style={draw, line width=1pt, postaction={on each segment={mid arrow=black}}},
  sregion/.style={fill, opacity=0.2},
  st/.style={fill=Fuchsia},
  su/.style={fill=YellowOrange},
  tu/.style={fill=ForestGreen},
  clabel/.style={fill=none, red}, 
  str/.style={<->}
}

\begin{document}

\def \dist {0.3cm}
\def \xdist {0.4cm}
\def \ydist {0.5cm}
\def \circSize {2.5pt} 
\def \armLen {0.4cm} 
\def \edgeShift {0.5mm} 
\def \wingLen {0.3cm} 
\def \sxdist {0.4cm}

\title{On the Induction of p-Cells}
\author{Lars Thorge Jensen, Leonardo Patimo}
\date{} 

\maketitle

\begin{abstract}  
  We study cells with respect to the $p$-canonical basis of the Hecke algebra of 
  a crystallographic Coxeter system (see \cite{JW, JeABC}) and their compatibility
  with standard parabolic subgroups. We show that after induction to the surrounding
  bigger Coxeter group the cell module of a right $p$-cell in a standard parabolic 
  subgroup decomposes as a direct sum of cell modules. Along the way, we state 
  some new positivity properties of the $p$-canonical basis.
\end{abstract}

\tableofcontents

\vspace{1cm}

\section{Introduction}

Let $(W,S)$ be a Coxeter system and let $\heck[(W,S)]$ denote its Hecke algebra.
In \cite{KL}, Kazhdan and Lusztig introduced a preorder on the elements of $W$ 
whose equivalence classes are known as right Kazhdan-Lusztig cells. These Kazhdan-Lusztig 
cells can be used to construct representations of $\heck[(W,S)]$, called cell modules.
Since their introduction, cells have been extensively studied and there is a very rich 
theory describing them.

If the Coxeter group is crystallographic, we can replace the Kazhdan-Lusztig basis of 
$\heck[(W,S)]$ with the $p$-canonical basis (see \cite{JW}) which can be obtained
as characters of the indecomposable parity complexes on a suitable Kac-Moody flag
variety with coefficients in a field of characteristic $p$. For $p=0$, the
$p$-canonical basis specializes to the Kazhdan-Lusztig basis by results of
H\"arterich \cite{HKMFlag}.

Working with the $p$-canonical basis,
we obtain a positive characteristic analogue of the Kazhdan-Lusztig cells, 
called $p$-cells. The motivation for studying $p$-cells is that, other than 
providing a first approximation of the multiplicative structure of the 
$p$-canonical basis, they allow one to construct representations of the 
Hecke algebra (called cell modules) that naturally come with a canonical basis. 
Moreover, according to a recent conjecture by Achar, Hardesty and Riche 
\cite{AHRCellConjs} there is a deep relation between $p$-cells in affine Weyl
groups and the geometry of nilpotent orbits.

The theory of $p$-cells has been initially developed by the first author in 
\cite{JeABC}, where several properties of Kazhdan-Lusztig cells are generalized to 
$p$-cells. The following theorem is inspired by \cite[Proposition 3.11]{BVPrimIdealsExcepGrps} and \cite[Proposition 1]{RIndResCells} and follows 
immediately from a key result in \cite[Theorem 3.9]{JeABC}.

\begin{thm*}[Parabolic restriction of right $p$-cells]
	Let $I \subseteq S$ be a finitary subset and $W_I \subseteq W$ the corresponding
   standard parabolic subgroup. Denote by $W^I$ the set of minimal length 
   representatives of cosets in $W/W_I$. 
   Let $D \subseteq W$ be a right $p$-cell of $W$ and $\heck[D]$ the corresponding
   right cell module of $\heck[(W, S)]$. Then $D$ is a disjoint union of sets of the
   form $xC$, where $x \in W^I$ and $C$ is a right $p$-cell for $W_I$.
   The direct sum of the associated right $p$-cells modules for $\heck[(W_I, I)]$ is
   isomorphic to the restricted module $\heck[D] \rvert_{\heck[(W_I, I)]}$.   
\end{thm*}

In this paper we deal with the generalization of another compatibility result with 
respect to parabolic subgroups. This property is  known as parabolic induction of 
cells. The following is the main result of the present paper.
\begin{thm*}[Parabolic induction for right $p$-cells]
   Let $I \subseteq S$ be a finitary subset and $W_I \subseteq W$ the corresponding
   standard parabolic subgroup. Denote by $\pre{I}{W}$ the set of minimal
   length representatives of cosets in $W_I \backslash W$. Let $C$ be a 
   right $p$-cell of $W_I$ and $\heck[C]$ the corresponding
   right cell module of $\heck[(W_I, I)]$.
   Then $C\cdot \pre{I}{W}$ is a union of right $p$-cells for $W$.
   The right module of $\heck[(W, S)]$ associated to $C\cdot \pre{I}{W}$
   is isomorphic to the induced module $\heck[C] \otimes_{\heck[(W_I, I)]} \heck[(W, S)]$.
\end{thm*}

Parabolic induction is well-known for Kazhdan-Lusztig cells. It was first proven by
Barbasch and Vogan in \cite[Proposition 3.15]{BVPrimIdealsExcepGrps} (see also
\cite[(5.26.1)]{LuRedGrps}) for a finite Weyl group $W$ in the setting of 
primitive ideals for complex semi-simple Lie algebras.
For arbitrary Coxeter groups, it was first conjectured by Roichman in 
\cite[\S5]{RIndResCells}. Finally, Geck proved it in the general 
setting of cells for Hecke algebras with unequal parameters in \cite{GeIndKLCells}.

In the following we will use the notation of the theorem.
One may introduce the $I$-hybrid basis, which combines the $p$-canonical basis 
of $\heck[(W_I, I)]$ and the standard basis elements corresponding to ${}^I W$ 
to get a $\Z[v, v^{-1}]$-basis of $\heck[(W, S)]$. Geck's proof is completely 
elementary and algebraic, combinatorial in nature. It is based on the study of 
the $I$-hybrid basis of the Hecke algebra (for $p = 0$) and its relation with 
the Kazhdan-Lusztig basis.

In contrast, our proof is based on a geometric argument.
We use the categorification of the $p$-canonical basis via indecomposable
parity complexes on the flag variety of a Kac-Moody group.
First, we prove that the base change coefficients between the $p$-canonical and
the $I$-hybrid basis are Laurent polynomials with non-negative coefficients. 
To achieve this, we relate the base change between these two bases to 
Braden's hyperbolic localization functors \cite{BrHypLoc}.
This works because parity sheaves are well behaved with respect to the 
hyperbolic localization functors (see \cite{JMW2}).
We remark that our argument is an adaptation of the work of Grojnowski and Haiman \cite{GHPosLLT}, which deals 
with mixed Hodge modules in characteristic $0$, to the setting of parity sheaves.

The rest of the proof follows along the lines of Geck's proof. It uses a hybrid
order which combines the right $p$-cell preoder in $W_I$ with the Bruhat order
on ${}^I W$. The crucial ingredient here is \cref{propPCanInIhybrid} which shows
that the base change coefficients between the $p$-canonical basis
and the $I$-hybrid basis are unitriangular with respect to the hybrid order.

\subsection{Structure of the Paper}

\begin{description}
   \item[\Cref{secBack}] We recall the definition of the Hecke algebra of a 
      crystallographic Coxeter system. Then we discuss its geometric categorification 
      via parity sheaves.  We also recall the definition of and the main result about
      Braden's hyperbolic localization functors. Finally, we introduce the 
      $p$-canonical basis and $p$-cells and mention some of their elementary properties.
   \item[\Cref{secInd}] We introduce a new basis, called $I$-hybrid basis, 
      for the Hecke algebra, combining the $p$-canonical basis for a parabolic subgroup 
      and the standard basis. Then we study positivity properties of the $p$-canonical
      basis with respect to the $I$-hybrid basis. Finally, we prove that induction for
      $p$-cells still works.
\end{description}

\subsection{Acknowledgements}

We would like to thank Jens Niklas Eberhardt, Simon Riche and Geordie Williamson 
for useful discussions. The first author has received funding from the European 
Research Council (ERC) under the European Union's Horizon 2020 research and 
innovation programme (grant agreement No 677147).

\section{Background}
\label{secBack}

\subsection{Crystallographic Coxeter Systems and their Hecke Algebras}
\label{secHeck}

Let $S$ be a finite set and $(m_{s, t})_{s, t \in S}$ be a matrix with entries in
$\N \cup \{\infty\}$ such that $m_{s,s} = 1$ and $ m_{s, t} = m_{t, s} \geqslant 2$ 
for all $s \neq t \in S$. Denote by $W$ the group generated by $S$ subject to the relations 
$(st)^{m_{s, t}} = 1$ for $s, t \in S$ with $m_{s, t} < \infty$. We say that $(W, S)$
is a \emph{Coxeter system} and $W$ is a \emph{Coxeter group}. The Coxeter group $W$
comes equipped with a length function $\ell: W \rightarrow \N$ and the Bruhat order
$\leqslant$ (see \cite{Hum} for more details). A Coxeter system $(W, S)$ is 
called \emph{crystallographic} if $m_{s, t} \in \{2, 3, 4, 6, \infty\}$ for 
all $s \neq t \in S$. We denote the identity of $W$ by $\Wid$. 

From now on, fix a generalized Cartan matrix $A = (a_{i,j})_{i, j \in J}$ (see 
\cite[\S1.1.1]{KuKMGrps}). Let $(J, \charlat, \{ \rt[i] : i \in J\}, 
\{ \cort[i] : i \in J \})$ be an associated Kac-Moody root datum (see 
\cite[\S1.2]{TiGrpKM} for the definition). Then $X$ is a finitely generated 
free abelian group, and for $i \in J$ we have elements $\rt[i]\in X$ and 
$\cort[i]\in\cocharlat = \Hom_{\Z}(\charlat, \Z)$ respectively satisfying                                                                           
$a_{i,j} = \cort[i](\rt[j])$ for all $i, j \in J$. 

We require our Kac-Moody root datum to satisfy the following assumptions
(see \cite[\S 7.3.1]{MarIntroKMGrps}):
\begin{enumerate}
   \item free, i.e. the set $\{ \rt[i] : i \in J\}$ is linearly independent over $\Z$ 
         in $\charlat$,
   \item cofree, i.e. the set $\{ \cort[i] : i \in J \}$ is linearly independent over $\Z$
         in $\cocharlat$, 
   \item cotorsion free, i.e. $\cocharlat / \sum_{i \in J} \Z \cort[i]$ is torsion-free,
   \item $\charlat$ is of rank $2\lvert J \rvert - \rk(A)$.
\end{enumerate}

\begin{remark}
   These assumptions were originally imposed by Mathieu and Kumar (see 
   \cite[Example 7.10]{MarIntroKMGrps} for the relation between a realization of $A$
   and a Kac-Moody root datum associated to $A$ satisfying the assumptions (i) - (iv)).
   Even though they are not necessary for the construction of Kac-Moody groups 
   (see \cite[Remarque 3.5 and \S3.19]{RoKMGrps} and \cite[\S8.7]{MarIntroKMGrps}),
   we will need the first assumption in our proof.
\end{remark}

To $A$ we associate a crystallographic Coxeter system $(W, S)$ as follows:
choose a set of simple reflections $S$ of cardinality $\lvert J \rvert$ and
fix a bijection $S \overset{\sim}{\rightarrow} J$, $s \mapsto i_s$. For $s \ne t \in S$
we define $m_{s, t}$ to be $2$, $3$, $4$, $6$, or $\infty$ if $a_{i_s, i_t} a_{i_t, i_s}$
is $0$, $1$, $2$, $3$, or $\geqslant 4$ respectively.
The group $W$ is called the \emph{Weyl group} of $A$.

The Hecke algebra $\heck = \heck[(W, S)]$ associated to $(W, S)$ is the 
free $\Z[v, v^{-1}]$-algebra with $\{ \std{w} \; \vert \; w \in W \}$ as basis, called the 
\emph{standard basis}, and multiplication determined by:
\begin{alignat*}{2}
   \std{s}^2 &= (v^{-1} - v) \std{s} + 1  \qquad && \text{for all } s \in S \text{,} \\
   \std{x} \std{y} &= \std{xy} && \text{if } \ell(x) + \ell(y) = \ell(xy) \text{.}
\end{alignat*}

There is a unique $\Z$-linear ring involution $\widebar{(-)}$ on $\heck$ satisfying
$\widebar{v} = v^{-1}$ and $\widebar{\std{x}} = \std{x^{-1}}^{-1}$. The Kazhdan-Lusztig 
basis element $\kl{x}$ is the unique element in $\std{x} + \sum_{y < x} v\Z[v] H_y$ 
that is invariant under $\widebar{(-)}$. This is Soergel's normalization from 
\cite{S2} of a basis introduced originally in \cite{KL}. 


Let $\iota$ be the $\Z[v, v^{-1}]$-linear anti-involution on $\heck$ satisfying 
$\iota(\std{s}) = \std{s}$ for $s \in S$ and thus $\iota(\std{x}) = \std{x^{-1}}$
for $x \in W$.

\subsection{Parity Complexes on Kac-Moody Flag Varieties}
\label{secParity}

In this section we want to recall some results about Kac-Moody flag varieties
and parity complexes on them. The standard references for Kac-Moody flag 
varieties are \cite{MarIntroKMGrps} and \cite{KuKMGrps}.


To our Kac-Moody root datum with Cartan matrix $A$ we can associate a (maximal) Kac-Moody group $G$
with a canonical torus  $T \subset G$.
Mathieu actually constructs group ind-schemes over $\Z$ in \cite{MaKMGrpsI} 
and \cite{MaKMGrpsII}, but for our purposes it is enough to consider their 
complex points as in \cite[\S6.1]{KuKMGrps} (see \cite[Exercise 8.123]{MarIntroKMGrps}
and \cite[\S3.20]{RoKMGrps} for comparisons of the two constructions).

Let $W$ be the Weyl group of $G$, i.e. the Weyl group of the Cartan matrix $A$.
Let $\rts$ denote the set of roots of $G$ and $\rts_+$ the set of positive roots. 
We denote by $\rts^{\text{re}}$ the set of real roots, i.e. roots that can be written 
as $w\cdot \alpha_i$, where $w\in W$ and $\alpha_i$ is a simple root.

Let  $U^+$ be the (pro-)unipotent group obtained as the image of the positive part 
of the associated Kac-Moody Lie algebra under the exponential map defined in
\cite[\S 6.1.1]{KuKMGrps} (also denoted $\mathfrak{U}^{ma+}$ in 
\cite[3.1]{RoKMGrps}). For every positive real root $\rt \in \rts^{\text{re}}_+$ there 
exists a one parameter subgroup $U_{\rt} \subset U^+$ isomorphic to the additive 
group $\mathbb{G}_a$ (see \cite[Example 6.1.5]{KuKMGrps}). 
For every $\rt$ we fix an isomorphism $u_{\rt}: \mathbb{G}_a 
\overset{\sim}{\longrightarrow} U_{\rt}$. Such an isomorphism is unique up to
multiplication by a unit in $\C$ (acting on $\mathbb{G}_a$). The subgroup $U^+$ 
is normalized by the torus $T$. Moreover, every $U_\alpha$ for $\alpha$ a positive 
real root is normalized by $T$ and we have $t u_\alpha(x) t^{-1} = u_{\alpha}(\alpha(t)x)$
for $t \in T$ and $x \in \C$.

The Borel subgroup $B$ is a subgroup of $G$ isomorphic to the semidirect 
product $B=T\ltimes U^+$. 
The set $G/B$ may be  equipped with the structure of an ind-projective ind-variety and it is called 
\emph{Kac-Moody flag variety} (see \cite[\S7.1]{KuKMGrps}).

For any finitary subset $I \subseteq S$ we have the corresponding standard parabolic 
subgroup $P_I$ containing $B$ (see \cite[Definitions 6.1.13 (4) and 6.1.18]{KuKMGrps}).
It admits a Levi decomposition $P_I = L_I \ltimes U_I$, where $L_I$ is a connected 
(finite dimensional) reductive group associated to $A_I$, the sub-Cartan matrix  of $A$ obtained 
by taking the rows and the columns indexed by $I$. We denote by $\rts_I$ the roots of $L_I$.
The group $T$ is a connected algebraic torus, whereas $G$, $B$, $U^+$, $P_I$ and $U_I$ are 
all pro-algebraic groups. The following two examples cover the most important cases:

\begin{ex}
   If $A$ is a Cartan matrix, then 
   the Kac-Moody root datum is equivalent to a root datum in the ordinary
   sense, and $G$ is the associated complex simply connected algebraic group, 
   $B$ a Borel subgroup, $T \subset B$ a maximal torus and $U^+$ the unipotent
   radical of $B$. In this case, $P_I$ is a standard parabolic subgroup and $L_I$ 
   the corresponding Levi.
\end{ex}

\begin{ex}
   Assume that $A= (a_{i,j})$ is a Cartan matrix of size $l-1$ and that the Kac-Moody
   root datum is simply connected (see \cite[Example 7.11]{MarIntroKMGrps}).   
   Let $G$ be the corresponding semi-simple simply connected algebraic group. 
   Moreover, one can add an $l$-th row and column to obtain a generalized 
   Cartan matrix $\widetilde{A}$ as follows
   \begin{align*}
      a_{j,l} &= -\theta(\cort[j]),\\
      a_{l,j} &= -\rt[j](\theta^{\vee}) \quad \text{for } 1 \leqslant j < l,\\
      a_{l,l} &= 2
   \end{align*}
   where the $\alpha_i$ are the simple roots of $G$ and $\theta$ the highest root.
   In this case, the Kac-Moody group $\widetilde{G}$ associated to $\widetilde{A}$
   is of so-called untwisted affine type (see \cite[\S 13.2]{KuKMGrps}) and some
   Kac-Moody flag varieties associated to $\widetilde{G}$ admit an 
   alternative description as partial affine flag varieties. Denote by 
   $\cal{K} = \C\llrrparen{t}$ the field of complex Laurent series
   and by $\cal{O} = \C\llrrbracket{t}$ the ring of complex power series. The Iwahori
   subgroup $I$ is defined as the inverse image of a Borel subgroup $B \subset G(\C)$ under
   the evaluation map $G(\cal{O}) \rightarrow G(\C)$, $t \mapsto 0$.
   Then the affine flag variety $G(\cal{K}) / I$ and the affine Grassmannian
   $G(\cal{K}) / G(\cal{O})$ are isomorphic to the Kac-Moody flag variety
   $\widetilde{G} / \widetilde{P}_I$ for $I= \emptyset$ and 
   $I = \{1, \dots, l-1\}$ respectively.
\end{ex}


We will mostly be interested in the orbits of $P_I$ on $\Fl[G] = G / B$. 
Fix a  finitary subset $I \subseteq S$ and the standard parabolic subgroup 
$W_I = \langle I \rangle \subseteq W$ generated by $I$.  The group $P_I$ can also 
be realized as $P=BW_IB$. Each left $W_I$-coset contains a unique element of minimal length. 
We denote by $\pre{I}{W}$ the set of these minimal coset representatives. 
The action of $P_I$ on $G/B$ induces the Bruhat decomposition
\[ G/B= \bigcup_{w \in {}^I W} {}^I Y_w\]
where ${}^IY_w:= P_I w B / B$. 
In the following, we will usually denote ${}^{\emptyset} Y_x$ by $Y_x$.
If $I = \emptyset$ each ${}^\emptyset Y_w$ is isomorphic to an affine space of dimension
$\ell(w)$ (see \cite[7.4.16 Proposition]{KuKMGrps}).
For $w \in {}^I W$ the decomposition of ${}^I Y_w$ into $B$-orbits gives 
a cell decomposition
\[ {}^I Y_w = \bigcup_{x \in W_I w} {}^{\emptyset} Y_x = 
              \bigcup_{x \in W_I w} \C^{l(x)} \text{.}\]

For $w\in W$ let $\rts(w)=\{\alpha\in\rts_+ \mid w^{-1}(\alpha)\in \rts^-\}$. 
Let $U_{\rts(w)}$ be the subgroup generated by $U_\alpha$, for $\alpha\in \rts(w)$.
Applying \cite[Lemma 6.1.3]{KuKMGrps} gives the following isomorphism (of varieties):
\[ U_{\rts(w)} \cong \prod_{\alpha\in \rts(w)} U_\alpha \text{.}\]
The multiplication induces an isomorphism (see \cite[Exercise 6.2.E (1)]{KuKMGrps} and \cite[Proposition 7.1.15]{KuKMGrps}) :
\begin{align}
   \label{KuBruhatReal}
   U_{\rts(w)} &\longrightarrow Y_w \text{,}\\
         u &\longmapsto uwB \text{.} \nonumber 
\end{align}

In this paper, we will view $\Fl[G] \defeq G / B$ as ind-variety equipped with
the algebraic statification coming from the $B$-orbits. Fix a field
$k$ of characteristic $p$. We will consider $D^b_{B}(\Fl[G])$, the
$B$-equivariant bounded constructible derived category of sheaves 
of $k$-vector spaces on $\Fl[G]$.

Recall the notion of a parity complex introduced by Juteau, Mautner and Williamson in  \cite[\S2.2]{JMW1}.
We will denote by $\parity[B](\Fl[G])$ the full subcategory of
$D^b_{B}(\Fl[G])$ whose objects are parity complexes.

\begin{thm}
   For each $w \in  W$ there exists up to isomorphism a unique
   indecomposable parity complex $\cal{E}(w) \in D^b_{B}(G/B)$ with support
   $\overline{Y_w}$ and restriction $\cal{E}(w)\vert_{Y_w} =
   \underline{k}_{Y_w} [\dim Y_w]$.
   Every indecomposable parity complex in $D^b_{B}(G/B)$ is up to shift isomorphic 
   to $\cal{E}(w)$ for some $w \in W$.
\end{thm}

The uniqueness up to isomorphism follows from 
\cite[Proposition 4.3 and Theorem 2.12]{JMW1}. 
The existence of the indecomposable parity complexes $\cal{E}(w)$ is shown in 
\cite[Theorem 4.6]{JMW1}.\footnote{In the original proof in \cite{JMW1} they 
require some mild assumptions on the characteristic of $k$ to prove the
existence of $\cal{E}(w)$. However, no assumption on the characteristic of $k$ 
is necessary in the $B$-equivariant setting (see for example 
\cite[paragraph following Lemma 9.6]{RWTiltPCan}). }
In the proof, the complex $\cal{E}(w)$ is constructed via the push-forward of 
the constant sheaf on a suitably chosen Bott-Samelson resolution.

\subsection{The \texorpdfstring{$p$}{p}-Canonical Basis and \texorpdfstring{$p$}{p}-Cells}

In this section, we recall the definition of the $p$-canonical basis in the
geometric setting. For this we will need to define the character map:

\begin{defn}
   Let $\cal{F}\in D^b_B(\Fl[G])$. We can define an element $\ch(\cal{F})\in \heck$ via 
   \[ \ch(\cal{F}) = \sum_{\substack{i\in \Z\\x\in W}} \dim H^{i}(\cal{F}_x)
        v^{-i-\ell(x)} \std{x} \]
   where $\cal{F}_x$ denotes the stalk of $\cal{F}$ in $xB/B$.
   The map $\ch$ is called the \emph{character map} and restricts to an isomorphism   
   $\ch: [\parity[B](\Fl[G])] \overset{\sim}{\longrightarrow} \heck$ of 
   $\Z[v, v^{-1}]$-algebras where $[\parity[B](\Fl[G])]$ denotes the split 
   Grothendieck group of $\parity[B](\Fl[G])$.
\end{defn}

\begin{defn}
   Define $\pcan{w} = \ch([\cal{E}(w)])$ for all $w \in W$. 
   The set $\{ \pcan{w} \; \vert \; w \in W\}$ is a $\Z[v, v^{-1}]$-basis of the
   Hecke algebra $\heck$, called the \emph{$p$-canonical or 
   $p$-Kazhdan-Lusztig basis}.  
\end{defn}

\begin{remark}
   Using the transposed generalized Cartan matrix as input for the diagrammatic
   category of Soergel bimodules and the definition of the $p$-canonical basis 
   given in \cite[Definition 2.4]{JeABC}, one obtains the same basis of 
   $\heck[(W, S)]$. This follows from the main result in \cite[Part 3]{RWTiltPCan}.
   In other words, the various definitions of the $p$-canonical basis are
   consistent (up to taking Langlands' dual input data).
\end{remark}

We will need the following positivity property of the $p$-canonical basis
which can be found in \cite[Proposition 4.2 and its proof]{JW}:

\begin{prop}
   \label{propPCanProps}
   For all $x \in W$ we have:
   \begin{enumerate}
      \item  $\pcan{x} = \std{x} + \sum_{y < x} \p{h}_{y, x} \std{y}$
            with $\p{h}_{y, x} \in \Z_{\geqslant 0}[v, v^{-1}]$,
      \item $\iota(\pcan{x}) = \pcan{x^{-1}}$ and thus in particular 
            $\pre{p}{h}_{y, x} = \pre{p}{h}_{y^{-1}, x^{-1}}$.
   \end{enumerate}
\end{prop}

The next result is obtained from \cite[Corollary 3.10]{JeABC} by applying
the anti-involution $\iota$ of $\heck$ and using \cref{propPCanProps} (ii).

\begin{prop}
   \label{corParPKLPols}
   Let $I \subseteq S$ be a finitary subset. Then for $x, y \in W_I$
   and $z \in \pre{I}{W}$ we have:
   \[ \pre{p}{h}_{yz, xz} = \pre{p}{h}_{y,x} \]
\end{prop}

Next, we recall the definition of $p$-cells from \cite[\S3.1]{JeABC}. This notion
is an obvious generalization of the notion of cells introduced by Kazhdan-Lusztig in \cite{KL}.

\begin{defn}
   For $h \in \heck$ we say that \emph{$\pcan{w}$ appears with non-zero coefficient
   in $h$} if the coefficient of $\pcan{w}$ is non-zero when expressing $h$ in
   the $p$-canonical basis.
   
   Define a preorder $\cle{R}$ (resp. $\cle{L}$) on $W$ as follows:
   $x \cle{R} y$ (resp. $x \cle{L} y$) if and only if $\pcan{x}$ appears with non-zero
   coefficient in $\pcan{y} h$ (resp. $h \pcan{y}$) for some $h \in \heck$.
   Right (resp. left) $p$-cells are the equivalence classes in $W$ with respect to
   $\cle{R}$ (resp. $\cle{L}$).
\end{defn}

For more properties and results about $p$-cells, we refer the reader to \cite{JeABC}.

\subsection{Hyperbolic Localization}

Let $T$ be a complex torus and $X$ a complex variety with an action of $T$.
In this section, we let $k$ be a field and $D(X)$ denote the constructible derived category
of sheaves of $k$-vector spaces on $X$. Moreover, we make the following 
assumption:
\begin{equation}
   \label{assTStableAffineCov}
   X \text{ has a covering by } T\text{-stable affine open subvarieties.}\tag{C}
\end{equation}
Note that this assumption is automatically satisfied if $X$ is normal by
Sumihiro's theorem (see \cite{SumEqComp, KKLV}).

Let $\chi: \mathbb{G}_m \longrightarrow T$ be a cocharacter of $T$. We want
to understand the hyperbolic localization with respect to the $\mathbb{G}_m$
action induced by $\chi$.

Denote by $Z \defeq X^{\chi} \subseteq X$ the variety of $\chi$-fixed points in $X$
and let $Z_1, \dots, Z_m$ be its connected components. For $1 \leqslant i \leqslant m$
we will denote the attracting and repelling varieties of the component $Z_i$ by
\[
   Z_i^+ = \{ x \in X \; \vert \; \lim_{s \rightarrow 0} \chi(s)x \in Z_i \} \qquad\text{ and}\qquad
   Z_i^- = \{ x \in X \; \vert \; \lim_{s \rightarrow \infty} \chi(s)x \in Z_i \}
\]
respectively. Let $Z^+$ (resp. $Z^-$) be the disjoint, disconnected union of the $Z_i^+$
(resp. $Z_i^-$) and define $f^{\pm}: Z \longrightarrow Z^{\pm}$ and 
$g^{\pm}: Z^{\pm} \rightarrow X$ via the component-wise inclusions. The attracting
and repelling maps $p^{\pm}: Z^{\pm} \longrightarrow Z$ are defined via 
$p^+(x) = \lim_{s\rightarrow 0} \chi(s)x$ and $p^+(x) = \lim_{s\rightarrow \infty} 
\chi(s)x$. It follows from \cite[Proposition 4.2]{Hess} that these are algebraic maps.

Braden defines the hyperbolic localization functors 
$(-)^{!\ast}, (-)^{\ast !}: D(X) \longrightarrow D(Z)$ associated to the cocharacter
$\chi$ for $\cal{F} \in D(X)$ as follows:
\begin{align*}
   \cal{F}^{! \ast} \defeq (f^+)^! (g^+)^{\ast} \cal{F}\text{,}\\
   \cal{F}^{\ast !} \defeq (f^-)^{\ast} (g^-)^! \cal{F}\text{.}
\end{align*}

We will use the following result (see \cite[Theorem 1]{BrHypLoc}):
\begin{thm}
   \label{thmHypLoc}
   For any $\cal{F} \in D(X)$ there is a natural morphism $\iota_{\cal{F}}: 
   \cal{F}^{\ast !} \longrightarrow \cal{F}^{! \ast}$. If $\cal{F}$ is weakly equivariant
   (e.g. comes from an object in the equivariant derived category), then
   \begin{enumerate}
      \item there are natural isomorphisms $\cal{F}^{! \ast} \cong (p^+)_! (g^+)^* \cal{F}$
            and $\cal{F}^{\ast !} \cong (p^-)_{\ast} (g^-)^! \cal{F}$ and
      \item $\iota_{\cal{F}}: \cal{F}^{\ast !} \longrightarrow \cal{F}^{! \ast}$ is
            an isomorphism.
   \end{enumerate}
\end{thm}

Using this result, Braden proves for $k = \Q$ that the hyperbolic localization
of the intersection cohomology complex $\mathbf{IC}(X; \Q)$ is a direct
sum of shifted intersection cohomology complexes.

\section{Induction of \texorpdfstring{$p$}{p}-Cells}
\label{secInd}

\subsection{Positivity Properties}
\label{secPosProps}

Recall that for our fixed finitary subset $I \subseteq S$ we denote the parabolic subgroup
generated by $I$ by $W_I$ and the minimal coset representatives by ${}^I W$.
The multiplication induces a bijection 
\begin{align*}
   W_I \times \pre{I}{W} &\overset{\sim}{\longrightarrow} W\\
    (x, y) &\longmapsto xy
\end{align*}
and we have $\ell(xy) = \ell(x) + \ell(y)$ if $(x, y) \in W_I \times {}^I W$ 
(see \cite[Proposition 2.4.4]{BBCoxGrps}). The following $I$-hybrid basis 
interpolates between the standard basis for $I = \diameter$ and the $p$-canonical
basis for $I = S$.\footnote{Technically, the case $I = S$ is only allowed if 
$W$ is finite.}

\begin{lem}
   The set $\{\mixCan{x}{y} \; \vert \; x \in W_I, y \in \pre{I}{W} \}$ is a 
   $\Z[v,v^{-1}]$-basis of the Hecke algebra $\heck$, called the \emph{$I$-hybrid basis}.
\end{lem}
\begin{proof}
	Notice that $\pcan{x}\std{y}\in \std{xy}+\sum_{z < xy} \N[v, v^{-1}] \std{z}$. 
    Since the base change matrix with the standard basis is unitriangular, the 
    set $\{\mixCan{x}{y} \; \vert \; x \in W_I, y \in \pre{I}{W} \}$ 
    is also a basis.
\end{proof}

We introduce the following notation for the base change coefficients between the
$p$-canonical and the $I$-hybrid bases
\[ \pcan{w} = \sum_{x \in W_I,\; y \in \pre{I}{W}}
      \p{r}^I_{xy,w} \mixCan{x}{y} \]
with $\p{r}^I_{xy,w} \in \Z[v, v^{-1}]$ for $w \in W$, $x \in W_I$ and $y \in \pre{I}{W}$.
The following result shows that these polynomials have in fact non-negative coefficients:

\begin{prop}
   \label{propPos}
   For all $w \in W$, $x \in W_I$ and $y \in \pre{I}{W}$ the polynomial $\p{r}^I_{xy, w}$ lies in 
   $\Z_{\geqslant 0}[v, v^{-1}]$. 
\end{prop}

The main idea is to replace the purity arguments in \cite[Theorem 3.2]{GHPosLLT} by
parity arguments. Before explaining the proof, we need to introduce some more notation.

Recall the decomposition $P_I=L_I\ltimes U_I$ from \S \ref{secParity}. To simplify
notation, we will write $L$ for $L_I$ and $P$ for $P_I$ from now on.
Denote by $B_L$ the Borel subgroup of $L$ induced
by our choice of positive roots, so that we have $B_L = B \cap L$.

\begin{lem}\label{yB}
For any $y \in \pre{I}{W}$ the stabilizer $\pre{y}{B}$ of $yB$ satisfies 
$B_L = \pre{y}{B} \cap L$.
\end{lem}
\begin{proof}

The group ${}^yB\cap L$ is the Borel subgroup of $L$ corresponding to the system of positive roots $\{\alpha \in \Phi_I \mid y^{-1}(\alpha) \in \Phi_I^+\}\subset \Phi_I$.
 
Since $y\in {}^IW$, we have $y^{-1}s > y^{-1}$ for all $s \in I$, so 
$y^{-1}(\rt[s])\in \rts^+$. Hence, for all $\rt \in \rts_I^+$ we have 
$y^{-1}(\rt)\in \rts^+$ and similarly for  $\rt \in \rts_I^-$ we have 
$y^{-1}(\rt)\in \rts^-$.  It follows that 
$\{\alpha \in \rts_I \mid y^{-1}(\alpha) \in \Phi_I^+\}=\rts_I^+$ and
$B_L={}^yB\cap L$.
\end{proof}

Hence, for any $y \in \pre{I}{W}$ we can realize 
$\Fl[L]$ inside $\Fl[G]$ via the isomorphism
\begin{align*}
i_y: \Fl[L] &\overset{\sim}{\longrightarrow} LyB/B \subseteq \Fl[G] \\
xB_L &\longmapsto xyB
\end{align*}


Denote by $X_w = \bigcup_{u \leqslant w} Y_u \subseteq \Fl[G]$ the 
Schubert variety associated to $w \in W$. 

We fix a dominant co-character  $\gamma: \C^{\ast} \longrightarrow T$  whose
stabilizer in $W$ is $W_I$ (this is possible because our realization is free). 
For each $w \in W$ the connected components of the
fixed locus $X_w^{\gamma}$ are precisely the intersections
$X_w \cap i_y(\Fl[L])$ for $y \in \pre{I}{W}$ (note that $\gamma$ acts trivially
by conjugation on a root subgroup $U_{\rt}$ if and only if $\rt$ lies in $\rts_I$).

\begin{lem}
	Let $y\in {}^IW.$ Then $\gamma$ retracts $ByB$ on $yB$, i.e.
	\[\lim_{t\rightarrow 0}\gamma(t) byB=yB\]
	for all $b\in B$.
\end{lem}
\begin{proof}
	Let $\rts(y)=\{\beta_1, \beta_2, \ldots, \beta_k\}$.
	As in \cref{yB}, we have $y^{-1}(\alpha)\in \rts^+$ for any $\alpha\in \rts_I^+$, hence $\rts(y)\subset \rts^+\setminus \rts^+_I$. In particular, we have $\langle \beta,\gamma\rangle >0$ for all $\beta \in \rts(y)$.

	By \cref{KuBruhatReal} we can write 
   $byB=u_{\beta_1}(x_1)u_{\beta_2}(x_2)\ldots u_{\beta_k}(x_k)yB$. Then 
	\[\gamma(t)u_{\beta_1}(x_1)u_{\beta_2}(x_2)\ldots u_{\beta_k}(x_k)yB = 
      u_{\beta_1}(t^{\langle \beta_1, \gamma \rangle}x_1)u_{\beta_2}(
         t^{\langle \beta_2, \gamma \rangle} x_2)\ldots 
      u_{\beta_k}(t^{\langle \beta_k, \gamma \rangle} x_k)yB\]
	and $\lim_{t\rightarrow 0} \gamma(t)byB=yB$.
\end{proof}

\begin{lem}
The attracting map $\pi_y: PyB/B \longrightarrow LyB/B$ is induced by the group
homomorphism $P\rightarrow L$. The attracting variety to $i_y(\Fl[L])$ is 
$PyB/B = \bigcup_{x \in W_I} Y_{xy}$. Moreover, $\pi_y$ is an affine bundle, with 
fiber isomorphic to $\C^{\ell(y)}$.

\end{lem}
\begin{proof}
	By the previous lemma we have $ByB\subset \pi_y^{-1}(yB)$. Since $W_I$ stabilizes 
   $\gamma$, the Levi subgroup $L$ commutes with the image of $\gamma$. We can write 
   any element $p\in P$ as $p=lu$, with $l\in L$ and $u\in U\subset B$. Now we have
	\[\pi_y(pyB)=\lim_{t\rightarrow 0} \gamma(t)pyB=\lim_{t\rightarrow 0} l \gamma(t)uyB =lyB.\]
	The following diagram of $L$-equivariant maps is commutative
	\[\begin{tikzcd}
      L \times_{B_L} ByB/B \arrow{r}{\sim} \arrow{d}{\pr_1} & 
         PyB/B \arrow{d}{\pi_y} \\
      \Fl[L] \arrow{r}{i_y}[swap]{\sim} & LyB/B
	\end{tikzcd}
	\]
	and the horizontal arrows are bijections since $\pi_y$ is an $L$-equivariant fiber bundle 
   and \[\pr_1^{-1}(yB/B) \cong \pi_y^{-1}(yB/B)\cong ByB\cong \C^{\ell(y)}.\qedhere\]
\end{proof}
Denote by $j_y: PyB/B \longrightarrow \Fl[G]$ the inclusion. The following result gives 
a geometric interpretation of the coefficients occuring when expressing the class of 
an object in $D^b_B(\Fl[G])$ in terms of the $I$-hybrid basis. 

\begin{lem}\label{mainlemma}
   For $x \in W_I$ and $y \in \pre{I}{W}$ we have:
   \begin{enumerate}
      \item For $\cal{F}\in D^b_{B_L}(\Fl[L])$ we have 
            \[\ch\left((j_{y})_! \pi_{y}^{\ast} (i_{y})_{\ast} \cal{F}\right)=
               v^{-\ell(y)}\ch(\cal{F})H_y.\]
            In this equation $\ch(\cal{F}) \in \heck[(W_I, I)]$ is viewed as an element
            in $\heck$ by identifying $\heck[(W_I, I)]$ as the subalgebra of $\heck$
            generated by $\std{s}$ for $s \in I$.
      \item For $\mathcal{F} \in D^b_B(\Fl[G])$ the coefficient of $\pcan{x} \std{y}$
            in $\ch(\mathcal{F})$ when expressed in the $I$-hybrid basis is $v^{\ell(y)}$
            times the coefficient of $\pcan{x}$ in $\ch\left(i_y^{\ast} (\pi_y)_! 
            j_y^{\ast} \mathcal{F}\right)$ when expressed in the $p$-canonical basis.
   \end{enumerate}
\end{lem}
\begin{proof}
	\begin{enumerate}
		\item We have $\left((j_{y})_! \pi_{y}^{\ast} (i_{y})_{\ast} \cal{F} \right)_z \neq 0$
            only if $z \in W_Iy$. For $x \in W_I$ we have 
            \[ \left( (j_{y})_! \pi_{y}^{\ast} (i_{y})_{\ast} \cal{F} \right)_{xy} = 
               \left( (i_y)_{*} \cal{F} \right)_{xy} = \cal{F}_{x}\]
		since $i_y$ is an isomorphism and $i_y^{-1}(xyB)=x B_L$. Hence, it follows that
      \begin{align*}
         \ch \left( (j_{y})_! \pi_{y}^{\ast} (i_{y})_{\ast} \cal{F} \right) &=
            \sum_{\substack{i \in \Z\\x \in W_I}} 
               v^{-i -\ell(x) -\ell(y)} \dim H^{i}(\cal{F}_{x}) \std{x}\std{y} \\
            &= v^{-\ell(y)} \ch(\cal{F}) \std{y}\text{.}
      \end{align*}
	\item The map $\pi_y$ is a topological fibration with fibers isomorphic to 
         $\C^{\ell(y)}$.
         Since $\cal{F}$ is constant along the fibers of $\pi_y$ we have 
         \[(j_{y})_! \pi_{y}^{\ast} (i_{y})_{\ast}  i_y^{\ast} (\pi_y)_! j_y^{\ast} \cal{F}
            \cong (j_{y})_! \pi_{y}^{\ast} (\pi_y)_! j_y^{\ast} \cal{F}
            \cong (j_{y})_! j_y^{\ast} \cal{F} [-2\ell(y)] \text{.} \]
         
         The flag variety $\Fl[G]$ is the disjoint union of $PyB/B$, for $y\in {}^I W$.
         It follows that
         \begin{align*}
               \ch(\cal{F}) &=\sum_{y\in {}^IW}\ch \left((j_{y})_! j_y^{\ast} \cal{F} \right) \\
                  &= \sum_{y\in {}^IW} v^{2\ell(y)} \ch \left({(j_{y})_! \pi_{y}^{\ast} 
                        (i_{y})_{\ast} i_y^{\ast} (\pi_y)_! j_y^{\ast} \cal{F}}\right)\\
                  &= \sum_{y\in {}^IW} v^{\ell(y)} \ch\left({ i_y^{\ast} (\pi_y)_! 
                        j_y^{\ast} \cal{F}}\right) H_y
         \end{align*}
         where we used the first part for the last equality.
         We conclude by expressing $\ch\left({i_y^{\ast} (\pi_y)_! j_y^{\ast} 
         \cal{F}} \right)$ in the $p$-canonical basis of $\heck[(W_I, I)].$
\end{enumerate}
\end{proof}

The following result crucially relies on Braden's hyperbolic localization:

\begin{prop}\label{paritytoparity}
   For any $y \in \pre{I}{W}$ the functor $i_y^{\ast}(\pi_y)_!j_y^{\ast}$ preserves
   parity complexes and thus restricts to a functor 
   $\parity[B](\Fl[G]) \longrightarrow \parity[B_L](\Fl[L])$.
\end{prop}
\begin{proof}
   \cref{thmHypLoc} implies that $\bigoplus_{y \in \pre{I}{W}} (\pi_y)_! j_y^{\ast}$ is 
   isomorphic to Braden's hyperbolic localization functor 
   $(-)^{!\ast}: D^b_{B}(\Fl[G]) \longrightarrow D^b_{B_L}(\Fl[G]^{\gamma})$.
   
   Juteau, Mautner and Williamson develop in \cite[\S2.2]{JMWParShTilt} a general
   framework to show that the hyperbolic localization of the pushforward of the
   constant sheaf on a resolution of singularities satisfying
   several properties (see \cite[\S2.2 (1) - (3)]{JMWParShTilt}) is a parity complex.
   In \cite[Proof of Theorem 1.6]{JMWParShTilt} they verify that these conditions
   are satisfied for Bott-Samelson resolutions of Schubert varieties of a 
   Kac-Moody group. Since any indecomposable parity complex $\cal{E}_w$ occurs as a 
   direct summand of the push-forward of the constant sheaf on a suitably chosen 
   Bott-Samelson resolution, it follows that $(\cal{E}_w)^{!\ast}$ is a parity complex.
   Finally, a complex $\cal{F}$ on $\Fl[G]^{\gamma}$ is parity if and only if
   its restriction to any connected component is a parity complex.
\end{proof}

\begin{proof}[Proof of \cref{propPos}.]
	By \cref{mainlemma}(ii), the polynomial ${}^pr_{xy,w}^I$ is equal to $v^{\ell(y)}$ the coefficient of $\pcan{x}$ in $\ch\left(i_y^*(\pi_y)_!j_y^*\cal{E}_w\right)$. By  \cref{paritytoparity}, the complex 
	 $i_y^*(\pi_y)_!j_y^*\cal{E}_w$ is a parity complex on $\Fl[L]$, hence its character has positive coefficients when written in the $p$-canonical basis of $\heck[(W_I, I)]$.
\end{proof}

\begin{remark}
   For $z \in W$, $x \in W_I$ and $y \in {}^I W$ write
   \[ \mixCan{x}{y} \cdot \pcan{z} = \sum_{u \in W_I,\; w \in \pre{I}{W}}
      \p{d}^{I, uw}_{xy, z} \;\mixCan{u}{w} \]
   with $\p{d}^{I, uw}_{xy,z} \in \Z[v, v^{-1}]$. Grojnowski and Haiman show in
   \cite[Corollary 3.9]{GHPosLLT} that the Laurent polynomials $\p{d}^{I, uw}_{xy,z}$
   have non-negative coefficients for $p=0$. Recently Williamson \cite{WReflSubgrps} obtained, after specializing $v$ to $1$, a more general results  which holds for a larger class of reflection subgroups of $W$.
\end{remark}

\subsection{Main Result}
\label{secMainResult}

The proof of the main result draws inspiration from \cite{GeIndKLCells}.
Throughout, we are working with the right $p$-cell preorder instead of the left
Kazhdan-Lusztig cell preorder. The analogous version for left $p$-cells
can be obtained by applying the anti-involution $\iota$. The following 
result is the analogue of \cite[Lemma 2.2]{GeIndKLCells}. Its proof 
works in our setting after replacing all Kazhdan-Lusztig related 
notions by the corresponding $p$-canonical analogues.
We will rewrite the proof here for the sake of completeness.

\begin{lem}
   \label{lemIdeal}
   Let $\mathcal{J} \subseteq W_I$ be a subset such that 
   \[ \{ u \in W_I \; \vert \; u \cle{R} w \text{ for some } w \in \mathcal{J} \} 
   \subseteq \mathcal{J}\text{.}\]
   Let $\mathcal{M} = \langle \pcan{x}\std{y} \; \vert \; x \in \mathcal{J}, 
         y \in \pre{I}{W} \rangle_{\Z[v, v^{-1}]} \subseteq \heck$. 
   Then $\mathcal{M}$ is a right ideal in $\heck$.
\end{lem}
\begin{proof}
   Since $\heck$ is generated by $\std{s}$ for $s \in S$ as an algebra, it is
   enough to check that $\pcan{x}\std{y}\std{s} \in \mathcal{M}$ for all
   $x \in W_I$ and $y \in \pre{I}{W}$. Deodhar's lemma (see 
   \cite[Lemma 2.1.2]{GPChFinCoxGrps}) shows that there are three cases to consider:
   \begin{enumerate}
      \item $ys \in \pre{I}{W}$ and $\ell(ys) > \ell(y)$. In this case, we have 
            $\pcan{x} \std{y} \std{s} = \pcan{x} \std{ys} \in \mathcal{M}$ as desired.
      \item $ys \in \pre{I}{W}$ and $\ell(ys) < \ell(y)$. Then
            $\pcan{x} \std{y} \std{s} = \pcan{x} (\std{ys} + (v^{-1} - v)\std{y}) 
            \in \mathcal{M}$.
      \item $ys \notin \pre{I}{W}$. For $t = ysy^{-1} \in I$ it follows that
            $\ell(ys) = \ell(y) + 1 = \ell(ty)$ and thus
            \[ \pcan{x} \std{y} \std{s} = \pcan{x} \std{ys}
               = \pcan{x} \std{ty} = \pcan{x} \std{t} \std{y} \text{.} \]
            The definition of right $p$-cells implies that $\pcan{x} \std{t}$
            is a linear combination of terms $\pcan{z}$ with $z \cle{R} x$.
            Our assumption then ensures that $\pcan{x} \std{y} \std{s}$ is a
            linear combination of terms $\pcan{u} \std{y}$ with $u \in \mathcal{J}$.
            This concludes the proof.\qedhere
   \end{enumerate}
\end{proof} 

As in \cite[\S3]{GeIndKLCells} we introduce the following hybrid preorder
on $W$:
\begin{defn}
   Let $x, u \in W_I$ and $y, w \in \pre{I}{W}$. We write $xy \mlt{I, R} uw$ if
   $x \cle{R} u$ in the right $p$-cell preorder and $y < w$ in the Bruhat order.
   We write as well $xy \mle{I, R} uw$ if $xy \mlt{I, R} uw$ or $xy = uw$.
\end{defn}

A crucial ingredient in Geck's proof is \cite[Proposition 3.3]{GeIndKLCells}.
For its proof Geck uses the characterization of the Kazhdan-Lusztig basis
element $\kl{w}$ as the unique self-dual element in $\std{x} + \sum_{y < x} v\Z[v] \std{y}$.
For this reason, his proof does not work for the $p$-canonical basis. Since Geck
works in the unequal parameter case, he cannot rely on positivity properties
which will allow us to conclude.

\begin{prop}
   \label{propPCanInIhybrid}
   We have for $x \in W_I$ and $y \in \pre{I}{W}$:
   \[ \pcan{xy} = \pcan{x} \std{y} + 
         \sum_{\substack{u \in W_I,\; w \in \pre{I}{W}\\uw \mlt{I, R} xy}}
         \p{r}^I_{uw, xy} \pcan{u} \std{w} \]
   In particular, the polynomial $\p{r}^I_{uw, xy}$ vanishes unless $uw \mle{I, R} xy$.
\end{prop}
\begin{proof}
   Let $x \in W_I$ and $y \in \pre{I}{W}$. Consider for 
   $\mathcal{J} = \{u \in W_I \; \vert \; u \cle{R} x\}$ the right ideal 
   $\mathcal{M} \subset \heck$ constructed as in \cref{lemIdeal}. Clearly, 
   $\pcan{x} \pcan{y}$ lies in $\mathcal{M}$ (as $\pcan{x} \in \mathcal{M}$).
   Since $\ell(xy) = \ell(x) + \ell(y)$, it follows that
   \begin{equation}
      \label{eqPCanProd}
      \pcan{x} \pcan{y} = \pcan{xy} + \sum_{z < xy} m_z \pcan{z} \text{.}
   \end{equation}
   for some $m_z\in \N[v,v^{-1}]$.
   \cref{propPos} implies that when expressing the $p$-canonical basis elements
   on the right hand side of \eqref{eqPCanProd} in the $I$-hybrid basis, there cannot be
   any cancellations. Therefore, we see that $\pcan{xy}$ lies in $\mathcal{M}$.   
   This means that we can express $\pcan{xy}$ as follows
   \begin{align*}
      \pcan{xy} &= \sum_{\substack{u \in W_I,\; w \in \pre{I}{W}\\u \cle{R} x}}
                   \p{r}^I_{uw, xy} \pcan{u} \std{w} 
   \end{align*}
   with $\p{r}^I_{uw, xy}\in \N[v,v^{-1}]$. It remains to show that  
   $\p{r}^I_{uw, xy}=0$ unless $uw = xy$ or $w<y$.
   Expanding in the standard basis (see \Cref{propPCanProps} (i)) we have
   \[
    \pcan{xy}= \sum_{\substack{z \leqslant u \in W_I,\; w \in \pre{I}{W}\\u \cle{R} x}}
                  \p{r}^I_{uw, xy} \p{h}_{z, u} \std{zw}.
   \]   
	By comparing coefficients in the standard basis we get
	\begin{equation}
      \label{eqBCCoeffs}
      \p{h}_{zw,xy}=\sum_{u\in W_I} \p{r}^I_{uw, xy} \p{h}_{z, u} \text{.}
   \end{equation}
	Since both the $p$-Kazhdan-Lusztig polynomials $\p{h}_{z, u}$ and the 
    polynomials $\p{r}^I_{uw, xy}$ have non-negative coefficients, 
    $\p{r}^I_{uw, xy}\neq 0$ for some $w \not \leq y$ and $u \in W_I$ implies 
    $\p{h}^I_{zw, xy}\neq 0$. But if $w \not\leq y$, then also 
   $zw\not\leq xy$ because the quotient map $W\rightarrow {}^IW$ is a strict morphism of posets (see \cite[Lemma 2.2]{DoInvForm}), 
   contradicting \Cref{propPCanProps}(i).
   
   Suppose $w=y$. \Cref{corParPKLPols} together with \eqref{eqBCCoeffs} implies that 
   $\p{r}_{xy,xy}^I = 1$ and that $\p{r}_{uy,xy}^I$ vanishes for $u < x$.
   This concludes the proof.
%
\end{proof}

\begin{cor}
   \label{corIdealPCanSpans}
   Let $\mathcal{J} \subseteq W_I$ and $\mathcal{M}$ be as in \cref{lemIdeal}.
   Then we have
   \[ \mathcal{M} = \langle \pcan{xy} \; \vert \; x \in \mathcal{J}, y \in \pre{I}{W}
      \rangle_{\Z[v, v^{-1}]} \text{.} \]
\end{cor}
\begin{proof}
   Let $x \in \mathcal{J}$ and $y \in \pre{I}{W}$. \cref{propPCanInIhybrid}
   shows that $\pcan{xy}$ lies in $\mathcal{M}$.

   Moreover, \cref{propPCanInIhybrid} shows that in any total order refining 
   the preorder $\mle{I, R}$ the base change matrix between the $p$-canonical 
   basis and the $I$-hybrid basis is uni-triangular Since its inverse is of 
   the same form, the claim follows.
\end{proof}

The proof of our main result is analogous to \cite[\S4]{GeIndKLCells}.
We will give a complete proof for the reader's convenience:

\begin{thm}
   \label{thmPreorderComp}
   Let $x, u \in W_I$ and $y, w \in \pre{I}{W}$. Then we have:
   \[ uw \cle{R} xy \quad \Rightarrow \quad u \cle{R} x \text{ in } W_I \]
   In particular, the following holds:
   \[ uw \ceq{R} xy \quad \Rightarrow \quad u \ceq{R} x \text{ in } W_I \]
\end{thm}
\begin{proof}
   It is enough to consider the case where $\pcan{uw}$ occurs with non-zero
   coefficient in $\pcan{xy} \std{s}$ for some $s \in S$.
   Consider the set $\mathcal{J} = \{z \in W_I \; \vert \; z \cle{R} x \}$.
   Since $\mathcal{J}$ satisfies the requirements in \cref{lemIdeal},
   it follows that $\mathcal{M} = \langle \pcan{a} \std{b} \; \vert \; 
   a \in \mathcal{J}, b \in \pre{I}{W} \rangle \subseteq \heck$ is a right ideal.
   
   \cref{corIdealPCanSpans} shows that $\pcan{xy}$ lies in $\mathcal{M}$.
   Since $\mathcal{M}$ is a right ideal, the element $\pcan{xy} \std{s}$ also
   lies in $\mathcal{M}$. From \cref{corIdealPCanSpans} it follows that we 
   can write
   \begin{equation}
      \label{eqPCanExp}
      \pcan{xy} \std{s} = \sum_{a \in \mathcal{J},\; b \in \pre{I}{W}} 
         \gamma_{xy, s}^{ab} \pcan{ab}
   \end{equation}
   with $\gamma_{xy, s}^{ab} \in \Z[v, v^{-1}]$ for all 
   $a \in \mathcal{J},\; b \in \pre{I}{W}$. Observe that the right-hand side
   of \eqref{eqPCanExp} is the expansion of $\pcan{xy} \std{s}$ in the $p$-canonical
   basis. Therefore, our assumption that $\pcan{uw}$ occurs with non-zero coefficient
   in $\pcan{xy} \std{s}$ means that $u \in \mathcal{J}$. This in turn implies
   $u \cle{R} x$ by the definition of $\mathcal{J}$ and finishes the proof of the
   theorem.
\end{proof}

Before we can state our main result, we should recall the definition of a cell module.
\begin{defn}
   For any right $p$-Cell $C \subseteq W$, write $w \cle{R} C$ (resp. $w \clt{R} C$) 
   if there exists an element $y \in C$ such that $w \cle{L} y$ (and $w \not\in C$). 
   By the definition of the right $p$-cell preorder, we can define right $\heck$-modules
   \[ \heck[{\cle{R} C}] \defeq \bigoplus_{w \cle{R} C} \Z[v, v^{-1}] \pcan{w} \]
   and similarly $\heck[{\clt{R} C}]$. Then the right $p$-cell module associated 
   to $C$ is defined as the right $\heck$-module given by the quotient 
   \[ \heck[C] \defeq \heck[{\cle{R} C}] / \heck[{\clt{R} C}] \text{.}\]
\end{defn}

\begin{thmlab}[Parabolic induction for right $p$-cells]
   Let $C$ be a right $p$-cell of $W_I$ and $\heck[C]$ the corresponding
   right cell module of $\heck[(W_I, I)]$.
   Then $C\cdot \pre{I}{W}$ is a union of right $p$-cells of $W$.
   The right module of $\heck[(W, S)]$ associated to $C\cdot \pre{I}{W}$
   is isomorphic to $\heck[C] \otimes_{\heck[(W_I, I)]} \heck[(W, S)]$.
\end{thmlab}
\begin{proof}
   The first part is a corollary of \Cref{thmPreorderComp} and the proof
   of the second part works exactly like the proof of 
   \cite[Lemma 5.2]{GeCellsConstrReps}.
\end{proof}

\subsection{An example: finite type \texorpdfstring{$C_3$}{C3} for \texorpdfstring{$p = 2$}{p = 2}}

In this section, we illustrate \cref{thmPreorderComp} in finite type $C_3$ 
where Kazhdan-Lusztig cells do not decompose into $p$-cells for $p = 2$ 
(see \cite[\S3.4.3]{JeABC}). We will use the following 
Cartan matrix as input and label the simple reflections accordingly. 
One may obtain a Kac-Moody root datum satisfying our assumptions 
from any based root datum of the corresponding connected semi-simple 
simply-connected algebraic group.

\[\begin{tikzpicture}[auto, baseline=(1.base)]
   \draw (1,0) -- (0,0);
   \draw (0,\edgeShift) -- (-1,\edgeShift);
   \draw (0,-\edgeShift) -- (-1,-\edgeShift);
   \path (0,0) to node[Greater] (mid) {} (-1,0);
   \draw (mid.center) to +(30:\wingLen);
   \draw (mid.center) to +(330:\wingLen);
   \node [DynNode] (1) at (-1,0) {$\color{red}{1}$};
   \node [DynNode] (2) at (0,0) {$\color{blue}{2}$};
   \node [DynNode] (3) at (1,0) {$\color{yellow}{3}$};
\end{tikzpicture} 
\Leftrightarrow \text{Cartan matrix: }
\begin{pmatrix}
   2 & -2 & 0 \\
   -1 & 2 & -1 \\
   0 & -1 & 2
\end{pmatrix}
\]
It should be noted that our input is Langlands' dual to the one given in
\cite[\S3.4.3]{JeABC} as we work on the geometric side and not with diagrammatic
Soergel bimodules.

Explicit computer calculation gives the following Kazhdan-Lusztig right cells
(as in \cite[\S3.4.3]{JeABC}):

\begin{align*}
   C_0 &= \{ \id \} \\
   C_1 &= \{ 1, 12, 121, 123 \} \\
   C_2 &= \{ 2, 21, 23, 212, 2123 \} \\
   C_3 &= \{ 3, 32, 321, 3212, 32123 \} \\
   C_4 &= \{ 13, 132, 1321 \}\\
   C_5 &= \{ 213, 2132, 21321 \} \\
   C_6 &= \{ 232, 2321, 23212 \} \\
   C_{7} &= \{ 2121, 21213, 212132, 2121321, 21213213 \} \\   
   C_8 &= \{ 1213, 12132, 121321 \} \\
   C_9 &= \{ 1232, 12321, 123212 \} \\
   C_{10} &= \{ 13212, 132123, 1213212, 1232123, 12132123 \} \\
   C_{11} &= \{ 21232, 212321, 2123212 \} \\
   C_{12} &= \{ 232123, 232121, 2321213, 23212132 \} \\
   C_{13} &= \{ w_0 \}
\end{align*}

For $p=2$ these right Kazhdan-Lusztig cells exhibit the following decomposition
behavior into right $p$-cells:
\begin{align*}
   C_2 &= \underbrace{\{ 2, 23\}}_{{}^2 C_{2'}} \cup \underbrace{\{21, 212, 2123 \}}_{{}^2 C_{2''}} \\
   C_3 &= \underbrace{\{ 3, 32\}}_{{}^2 C_{3'}} \cup \underbrace{\{ 321, 3212, 32123 \}}_{{}^2 C_{3''}} \\
   C_6 \cup C_{12} &= \underbrace{\{ 232 \}}_{{}^2 C_6} \cup \underbrace{\{ 2321, 23212, 232123 \}}_{{}^2 C_{6/12}}
                     \cup \underbrace{\{ 232121, 2321213, 23212132 \}}_{{}^2 C_{12}} \\
   C_i &= {}^2 C_i \text{ for } i \in \{0, \dots, 13\} \setminus \{2,3,6,12\}
\end{align*}

The Hasse-diagrams of the cell preorders look as follows. We display Kazhdan-Lusztig
right cells on the left and right $p$-cells on the right. In these diagrams the cells 
that are depicted at one height form a two-sided cell.
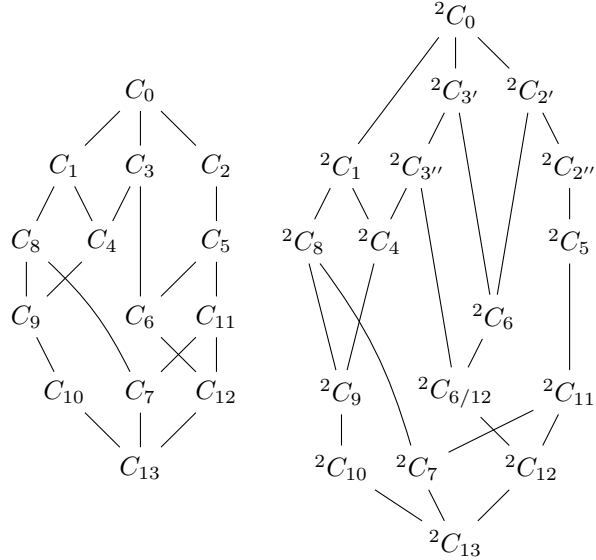
\[
\begin{tikzpicture}[auto, baseline=(current  bounding  box.center)]
   \node (0) at (0, 2.5) {$C_0$};
   \node (1) at (-1, 1.5) {$C_1$};
   \node (2) at (1, 1.5) {$C_2$};
   \node (3) at (0, 1.5) {$C_3$};
   \node (4) at (-0.5, 0.5) {$C_4$};
   \node (5) at (1, 0.5) {$C_5$};
   \node (8) at (-1.5, 0.5) {$C_8$};
   \node (9) at (-1.5, -0.5) {$C_9$};
   \node (11) at (1, -0.5) {$C_{11}$};
   \node (6) at (0, -0.5) {$C_6$};
   \node (10) at (-1, -1.5) {$C_{10}$};
   \node (7) at (0, -1.5) {$C_7$};
   \node (12) at (1, -1.5) {$C_{12}$};
   \node (13) at (0, -2.5) {$C_{13}$};
   \draw (0) to (1);
   \draw (0) to (2);
   \draw (0) to (3);
   \draw (1) to (8);
   \draw (1) to (4);
   \draw (3) to (4);
   \draw (3) to (6);
   \draw (2) to (5);
   \draw (4) to (9);
   \draw (5) to (6);
   \draw (5) to (11);
   \draw (8) to (9);
   \draw[bend left=10] (8) to (7);
   \draw (6) to (12);
   \draw (9) to (10);
   \draw (11) to (12);
   \draw (11) to (7);
   \draw (10) to (13);
   \draw (7) to (13);
   \draw (12) to (13);
\end{tikzpicture}
\quad
\begin{tikzpicture}[auto, baseline=(current  bounding  box.center)]
   \node (0) at (0, 3.5) {${}^2 C_0$};
   \node (3') at (0, 2.5) {${}^2 C_{3'}$};
   \node (2') at (1, 2.5) {${}^2 C_{2'}$};
   \node (1) at (-1.5, 1.5) {${}^2 C_1$};
   \node (3'') at (-0.5, 1.5) {${}^2 C_{3''}$};
   \node (2'') at (1.5, 1.5) {${}^2 C_{2''}$};
   \node (8) at (-2.0, 0.5) {${}^2 C_8$};
   \node (4) at (-1.0, 0.5) {${}^2 C_4$};
   \node (5) at (1.5, 0.5) {${}^2 C_5$};   
   \node (6) at (0.5, -0.5) {${}^2 C_6$};
   \node (9) at (-1.5, -1.5) {${}^2 C_9$};
   \node (6/12) at (0, -1.5) {${}^2 C_{6/12}$};
   \node (11) at (1.5, -1.5) {${}^2 C_{11}$};
   \node (10) at (-1.5, -2.5) {${}^2 C_{10}$};
   \node (7) at (-0.5, -2.5) {${}^2 C_7$};
   \node (12) at (1, -2.5) {${}^2 C_{12}$};
   \node (13) at (0, -3.5) {${}^2 C_{13}$};
   \draw (0) to (1);
   \draw (0) to (2');
   \draw (0) to (3');
   \draw (3') to (3'');
   \draw (3') to (6);
   \draw (2') to (2'');
   \draw (2') to (6);
   \draw (1) to (8);
   \draw (1) to (4);
   \draw (3'') to (4);
   \draw (3'') to (6/12);
   \draw (2'') to (5);
   \draw (4) to (9);
   \draw (5) to (11);
   \draw (8) to (9);
   \draw[bend left=10] (8) to (7);
   \draw (6) to (6/12);
   \draw (6/12) to (12);
   \draw (9) to (10);
   \draw (11) to (12);
   \draw (11) to (7);
   \draw (10) to (13);
   \draw (7) to (13);
   \draw (12) to (13);
\end{tikzpicture}\]

We will choose $I = \{1, 2\}$ so that $W_I$ is a finite Coxeter group of type $B_2$
and we get \[{}^I W = \{\id, 3, 32, 321, 3212, 32123 \}\text{.}\]
For $p = 2$ the Kazhdan-Lusztig right cell $\{2, 21, 212\}$ in $W_I$ decomposes
into two right $p$-cells $\{2\} \cup \{21, 212\}$ (see \cite[\S 3.4.1]{JeABC}).

For the right Kazhdan-Lusztig cells $C = \{2\}$ and $C' = \{21, 212\}$ in $W_I$ we get:
\begin{align*}
  C \cdot {}^I W = &\underbrace{\{ 2, 23\}}_{{}^2 C_{2'}} \cup \underbrace{\{232 \}}_{{}^2 C_{6}} 
		\cup \underbrace{\{2321, 23212, 232123 \}}_{{}^2 C_{6/12}} \\
  C' \cdot {}^I W = &\underbrace{\{ 21, 212, 2123\}}_{{}^2 C_{2''}} \cup \underbrace{\{213, 2132, 21321 \}}_{{}^2 C_{5}} 
		\cup \underbrace{\{232121, 2321213, 23212132 \}}_{{}^2 C_{12}} \\
		&\cup \underbrace{\{21232, 212321, 2123212 \}}_{{}^2 C_{11}}
\end{align*}

Observe that neither $C \cdot {}^I W$ nor $C' \cdot {}^I W$ is a union of
right Kazhdan-Lusztig cells.

\printbibliography

\Address

\end{document}